\documentclass{amsart}
\usepackage{hyperref}
\usepackage{amsthm,mathtools}
\usepackage[all]{xy}
\usepackage[vcentermath]{youngtab}
\usepackage{booktabs}

\theoremstyle{plain}
\theoremstyle{definition}
\newtheorem{theorem}{Theorem}[section]\newtheorem{lemma}[theorem]{Lemma}\newtheorem{proposition}[theorem]{Proposition}\newtheorem{corollary}[theorem]{Corollary}\newtheorem{example}[theorem]{Example}

\usepackage[dvipsnames]{xcolor}

\usepackage{amssymb,amsmath,tabularx,graphicx,float,placeins}
\usepackage{tikz}
\usetikzlibrary[calc,intersections,through,backgrounds,arrows,decorations.pathmorphing]

\def\Ocal{\mathcal{O}}

\def\Cbb{\mathbb{C}}\def\Qbb{\mathbb{Q}}\def\Zbb{\mathbb{Z}}

\def\Sfrak{\mathfrak{S}}

\def\veps{\varepsilon}

\def\pr{\prime}\def\ra{\rightarrow}

\newcommand{\N}{\mathbb{N}}

\DeclareMathOperator{\amaj}{amaj}
\DeclareMathOperator{\col}{col}

\DeclareMathOperator{\DEF}{def}
\DeclareMathOperator{\End}{End}

\DeclareMathOperator{\fmaj}{fmaj}
\DeclareMathOperator{\GL}{GL}
\DeclareMathOperator{\ind}{ind}

\DeclareMathOperator{\inv}{inv}

\DeclareMathOperator{\maj}{maj}
\DeclareMathOperator{\nneg}{neg}
\DeclareMathOperator{\Neg}{Neg}
\DeclareMathOperator{\reg}{reg}

\DeclareMathOperator{\refl}{refl}

\DeclareMathOperator{\rot}{rot}

\DeclareMathOperator{\sign}{sign}
\DeclareMathOperator{\sneg}{sneg}

\DeclareMathOperator{\stat}{stat}

\DeclareMathOperator{\triv}{triv}

\usepackage{todonotes} 
\usepackage{wasysym} 

\begin{document}

\title{Determinantal formulas with major indices}
\author{Thomas McConville, Donald Robertson, Clifford Smyth}
\date{\today}

\maketitle

\begin{abstract}
We give a simple proof of a major index determinant formula in the symmetric group discovered by Krattenthaler and first proved by Thibon using noncommutative symmetric functions. We do so by proving a factorization of an element in the group ring of the symmetric group. By applying similar methods to the groups of signed permutations and colored permutations, we prove determinant formulas in these groups as conjectured by Krattenthaler.
\end{abstract}

\section{Introduction}

Let $\Sfrak_n$ be the group of permutations of $[n]:=\{1,\ldots,n\}$. We often consider permutations in one--line notation $w=w_1\cdots w_n$ where $w_i=w(i)$. An integer $i\in[n-1]$ is a \emph{descent} of a permutation $w$ if $w_i>w_{i+1}$. The \emph{major index} $\maj(w)$ is the sum of the descents of $w$. For example, $\maj(314652)=1+4+5=10$. The \emph{major index matrix} is $(q^{\maj(uv^{-1})})_{u,v\in\Sfrak_n}$. In his survey of determinantal formulas, Krattenthaler discovered and communicated a proof by Thibon of the following identity.

\begin{theorem}[Theorem 56, \cite{krattenthaler2001advanced}]\label{thm:main_maj}
  For all $n\geq 1$,
\[ \det\left( q^{\maj(uv^{-1})} \right)_{u,v\in\Sfrak_n} = \prod_{k=2}^{n}(1-q^{k})^{n!\cdot (k-1)/k }.  \]  
\end{theorem}

\begin{example}
  The identity is trivial if $n=1$. For $n=2$, Theorem~\ref{thm:main_maj} gives
  \[ \det\left(\begin{matrix}1 & q\\q & 1\end{matrix}\right) = (1-q^2). \]
For $n=3$, if we index the rows and columns by $(123,\ 132,\ 213,\ 231,\ 312,\ 321)$, then 
\[ \det\left(\begin{matrix}
  1 & q^2 & q & q & q^2 & q^3\\
  q^2 & 1 & q & q & q^3 & q^2\\
  q & q^2 & 1 & q^3 & q^2 & q\\
  q^2 & q & q^3 & 1 & q & q^2\\
  q & q^3 & q^2 & q^2 & 1 & q\\
  q^3 & q & q^2 & q^2 & q & 1
  \end{matrix}\right) = (1-q^2)^3(1-q^3)^4. \]
\end{example}

To prove Theorem~\ref{thm:main_maj}, Thibon explicitly determined the eigenvalues of the major index matrix with multiplicity using the theory of noncommutative symmetric functions developed in \cite{krob1997noncommutative}. Stanley also determined the eigenvalues with multiplicity in \cite[Theorem 2.2]{stanley2001generalized} by applying a theorem of Bidigare, Hanlon, and Rockmore \cite[Theorem 1.2]{bidigare1999combinatorial}.

In this paper, we present a new, simpler proof of Theorem~\ref{thm:main_maj}. Our proof relies on a clever interpretation of the major index of a permutation given by Adin and Roichman in \cite{adin2001flag}, which we recall here. Let $t_k=(k,k-1,\dots,1)$ be a $k$-cycle for $2\leq k\leq n$. Each $w$ in $\Sfrak_n$ can be uniquely expressed in the form $t_n^{c_n}t_{n-1}^{c_{n-1}}\cdots t_2^{c_2}$ where $0\leq c_k < k$, and the major index of $w$ is $c_n+c_{n-1}+\cdots+c_2$. This means that the sequence $(t_n,t_{n-1},\ldots,t_2)$ is a perfect basis of $\Sfrak_n$, the definition of which we recall in Section~\ref{sec:bases}. This perfect basis determines a factorization of the major index matrix, which we use to evaluate its determinant in Section~\ref{sec:maj}.

Our proof of Theorem~\ref{thm:main_maj} was motivated by Zagier's proof \cite{zagier1992realizability} of the identity
\[ \det\left(q^{\inv(uv^{-1})}\right)_{u,v\in\Sfrak_n} = \prod_{k=2}^{n}(1-q^{k^2-k})^{n!\cdot(n-k+1)/(k^2-k)}, \]
where $\inv(w)$ is the number of inversions of a permutation $w$. Zagier considered an element of the group algebra $\Cbb(q)\Sfrak_n$ whose image under the regular representation is the matrix $(q^{\inv(uv^{-1})})$. By factoring this element of the group algebra, he obtained a corresponding factorization of the matrix $(q^{\inv(uv^{-1})})_{u,v \in \Sfrak_n}$ for which the determinants of the factors could be readily evaluated.

A \emph{colored permutation} $(w,x)$ consists of a permutation $w\in\Sfrak_n$ and $x~\in~(\Zbb/m\Zbb)^n$. We recall the group structure on colored permutations in Section~\ref{sec:fmaj}. When $m=2$, this group is isomorphic to the group of signed permutations, the real reflection group of type $B_n$. 

Based on extensive computational evidence, Krattenthaler conjectured in \cite{krattenthaler2005advanced} several analogues of Theorem~\ref{thm:main_maj} for colored permutations using variations on the major index given in \cite{adin2001descent,adin2001flag,reiner1993signed}. We prove all of his conjectured formulas and more in Sections~\ref{sec:fmaj},~\ref{sec:amaj}, and~\ref{sec:signed}. For each case, we construct a basis such that the relevant statistics can be read from the exponent vectors.

Theorem~\ref{thm:main_maj} and the various extensions we consider for colored or signed permutations are all specializations of group determinants. For a finite group $G$, its \emph{group determinant} is $\det(r_{gh^{-1}})_{g,h\in G}$ where $\{r_g\mid g\in G\}$ is a set of elements of a commutative ring. In his pioneering work on the representation theory of finite groups, Frobenius proved that if $\{r_g\mid g\in G\}$ is a set of indeterminates in a polynomial ring over $\Cbb$, then the irreducible factors of the group determinant naturally correspond to irreducible representations of $G$; see \cite{hawkins1978creation}.

We consider examples of the group determinant of the form $\det(q^{\stat(gh^{-1})})_{g,h\in G}$ for some statistic on $G$. We are especially intrigued by examples for which this determinant is a product of binomials. This behavior was proved for the length statistic on finite Coxeter groups in \cite{varchenko1993bilinear}, extending the aforementioned result from \cite{zagier1992realizability}.

The rest of this paper is structured as follows. Preliminary results on group determinants and perfect bases are given in Sections~\ref{sec:group_det} and~\ref{sec:bases}. Theorem~\ref{thm:main_maj} is proved in Section~\ref{sec:maj}. In Section~\ref{sec:fmaj}, we recall the flag major index on colored permutations introduced by Adin and Roichman in~\cite{adin2001flag} and prove a formula for the corresponding group determinant, answering \cite[Problem 49]{krattenthaler2005advanced}. In Section~\ref{sec:amaj}, we consider another statistic on colored permutations that we call the absolute flag major index, and we generalize and prove \cite[Conjecture 48]{krattenthaler2005advanced}. Finally, in Section~\ref{sec:signed} we prove an identity from which we can derive proofs of Conjectures~46,~47, and~50 in \cite{krattenthaler2005advanced}.

\section{Group determinants}\label{sec:group_det}

Let $G$ be a finite group and let $R$ be a commutative ring with $1$. The \emph{group ring} $RG$ is the free $R$-module with a distinguished basis that we identify with the elements of $G$. Multiplication of basis elements is the same as in $G$ and is extended linearly to $RG$.

Given a complex vector space $V\cong\Cbb^m$, a \emph{representation} is a group homomorphism $G\ra\GL(V)$, which we may extend to a ring homomorphism $RG\ra\End(V)$. For any element $\alpha=\sum_{g\in G}r_g g\in RG$ and any representation $\phi$, we set
\[ \Delta_{\phi}(\alpha) := \det\phi(\alpha) = \det\left(\sum_{g\in G}r_g\phi(g)\right). \]

\begin{example}
\leavevmode
\begin{enumerate}
\item
If $\phi_{\triv}:G\ra\GL(\Cbb^1)$ is the trivial representation, then $\Delta_{\phi_{\triv}}(\alpha)=\sum_{g\in G}r_g$.
\item 
If $G$ acts on a finite set $X$, then the \emph{permutation representation} $\phi_X:G\ra\Cbb^X$ assigns to an element $g$ the transformation $x\mapsto g\cdot x$ for $x\in X$. 
So $\phi_X(g)$ is a permutation matrix, i.e. it is a $0,1$--matrix such that every row and column has exactly one $1$.
\item 
The \emph{regular representation} $\phi_{\reg}$ is the permutation representation induced by the action of $G$ on itself by left multiplication. 
The \emph{group determinant} of $G$ is
\[ \Delta_{\phi_{\reg}}(\alpha) = \det\left(r_{uv^{-1}}\right)_{u,v\in G} \]
where $\alpha=\sum r_g g$.
\end{enumerate}
\end{example}
Two representations $\phi,\eta$ are \emph{equivalent} if there exists an invertible matrix $U$ such that $\eta(g)=U\phi(g)U^{-1}$ for all $g\in G$. 
It is clear that $\Delta_{\phi}(\alpha)=\Delta_{\eta}(\alpha)$ whenever $\phi$ and $\eta$ are equivalent representations.
The direct sum $\phi\oplus\eta$ of two representations satisfies $\Delta_{\phi\oplus\eta}(\alpha) = \Delta_{\phi}(\alpha)\Delta_{\eta}(\alpha)$. 
As any representation is equivalent to a direct sum of irreducible representations, we can always factor the determinant $\Delta_{\phi}(\alpha)$ as a product over the irreducible direct summands of $\phi$.


For an $m\times m$ matrix $M$, let $\theta_M(q)=\det(I-qM)$. 
For permutation representations, we have the following result.

\begin{proposition}
\label{prop:det_permrep}
  Let $G$ act on a finite set $X$, and let $\phi=\phi_X$ be the corresponding permutation representation. Fix $g\in G$, and let $\Ocal_1,\ldots,\Ocal_N$ be the orbits of the cyclic subgroup $\langle g\rangle$. Then
  \[ \theta_{\phi(g)}(q) = \prod_{k=1}^N(1-q^{|\Ocal_k|}). \]
\end{proposition}

\begin{proof}
  For $k\in[N]$, let $T_k$ be the $|\Ocal_k|\times|\Ocal_k|$ permutation matrix with $(i,j)$-entry equal to $1$ if $i=j+1\bmod{|\Ocal_k|}$. Up to a simultaneous permutation of rows and columns, the matrix $\phi(g)$ is the direct sum $T_1\oplus\cdots\oplus T_N$. Hence,
  \[ \theta_{\phi(g)}(q) = \det(I-q\phi(g)) = \prod_{k=1}^N\det(I-q T_k) = \prod_{k=1}^N(1-q^{|\Ocal_k|}).\qedhere \]
\end{proof}

Let $o(g)$ denote the order of an element $g$. In the regular representation, every orbit of $\langle g\rangle$ has size equal to $o(g)$. We immediately deduce the following corollary of Proposition~\ref{prop:det_permrep}.

\begin{corollary}\label{cor:det_reg}
  If $\phi=\phi_{\reg}$ is the regular representation of $G$, then
  \[ \theta_{\phi(g)}(q) = (1-q^{o(g)})^{|G|/o(g)}. \]
\end{corollary}

\section{Perfect bases}\label{sec:bases}

We follow the terminology of \cite{Shwartz2007MajorIA} for bases of groups.

Let $G$ be a finite group. A sequence $(g_1,\ldots,g_n)$ of elements of $G$ is a \emph{basis} if there exist positive integers $m_1,\ldots,m_n$ such that every element $g\in G$ may be uniquely expressed in the form $g = g_1^{c_1}\cdots g_n^{c_n}$ where $0\leq c_i< m_i$ for $i\in[n]$. This basis is \emph{perfect} if $m_i=o(g_i)$ for all $i$. A group admits a perfect basis if and only if it may be identified as a set with a Cartesian product of cyclic groups by the following map.
\begin{align*}
  \langle g_1\rangle\times\cdots\times \langle g_n\rangle &\stackrel{\sim}{\to} G\\
  (g_1^{c_1},\ldots,g_n^{c_n}) &\mapsto g_1^{c_1}\cdots g_n^{c_n}
\end{align*}
This map is not necessarily a group isomorphism.

\begin{example}\label{ex:S3_basis}
  Let $G=\Sfrak_3$. Take $g_1=(123),\ g_2=(12)(3)$, written in cycle notation. Then $(g_1,g_2)$ is a perfect basis since every element of $\Sfrak_3$ is uniquely expressible in the form $g_1^{c_1}g_2^{c_2}$ for $0\leq c_1<3,\ 0\leq c_2<2$.
  \begin{center}
    \begin{tabular}{l l}
      \hspace{1cm} $g_1^0g_2^0 = (1)(2)(3)$ & \hspace{1cm} $g_1^0g_2^1 = (12)(3)$\\
      \hspace{1cm} $g_1^1g_2^0 = (123)$ & \hspace{1cm} $g_1^1g_2^1 = (13)(2)$\\
      \hspace{1cm} $g_1^2g_2^0 = (132)$ & \hspace{1cm} $g_1^2g_2^1 = (1)(23)$
    \end{tabular}    
  \end{center}

\end{example}

\begin{example}\label{ex:Dn_basis}
  We extend Example~\ref{ex:S3_basis} to any dihedral group of order $2n$ for $n\geq 3$. Consider the dihedral group $G$ of isometries of a regular $n$--gon whose vertices are labeled $1,2,\ldots,n$ in clockwise order. As a group of permutations of $[n]$, $G$ is generated by a rotation $g_1=(123\cdots n)$ and a reflection $g_2=(1,n-1)(2,n-2)\cdots(\lfloor\frac{n-1}{2}\rfloor,\lfloor\frac{n+2}{2}\rfloor)(n)$. Every rotation symmetry is of the form $g_1^k$ for some $0\leq k<n$, and every reflection symmetry is of the form $g_1^kg_2$ for some $0\leq k<n$. Hence, $(g_1,g_2)$ is a perfect basis.
\end{example}

For the remainder of this section, we let $R=\Qbb[x_1,\ldots,x_n]$ and we consider the group ring $RG$.

\begin{lemma}\label{lem:factor_basis}
  If $(g_1,\ldots,g_n)$ is a basis of $G$, then we have the identity
  \[ \sum_{\substack{g=g_1^{c_1}\cdots g_n^{c_n}\\0\leq c_i<m_i}} x_1^{c_1}\cdots x_n^{c_n}\cdot g = \prod_{i=1}^n (1+x_ig_i+\cdots+x_i^{m_i-1}g_i^{m_i-1}). \]
  If $(g_1,\ldots,g_n)$ is a perfect basis, then
\begin{equation}
\label{eqn:alphaBetaFactor}
\prod_{i=1}^n (1+x_ig_i+\cdots+x_i^{m_i-1}g_i^{m_i-1}) (1-x_ig_i) = \prod_{i=1}^n(1-x_i^{m_i})\cdot 1. 
\end{equation}
\end{lemma}

\begin{proof}
  The first statement immediately follows by expanding the right hand side. The latter statement follows from the assumption $g_i^{m_i}=1$.
\end{proof}

\begin{theorem}\label{thm:general_basis}
  Suppose $G$ has a perfect basis $(g_1,\ldots,g_n)$, and set
  \[ \alpha = \sum_{\substack{g=g_1^{c_1}\cdots g_n^{c_n}\\0\leq c_i<m_i}} x_1^{c_1}\cdots x_n^{c_n}\cdot g. \]
  If $V\cong\Cbb^r$ is a vector space and $\phi:G\ra\GL(V)$ is a representation of $G$, then
  \[ \Delta_{\phi}(\alpha) = \prod_{i=1}^{n}\frac{(1-x_i^{m_i})^r}{\theta_{\phi(g_i)}(x_i)}. \]
\end{theorem}

\begin{proof}
  By Lemma~\ref{lem:factor_basis}, we have
  \[ \alpha\prod_{i=1}^n(1-x_ig_i) = \prod_{i=1}^n(1-x_i^{m_i})\cdot 1. \]
  Hence,
  \begin{align*}
    \det(\phi(\alpha)) &= \frac{\det\phi\left(\prod_{i=1}^n(1-x_i^{m_i})1 \right)}{\det\phi\left(\prod_{i=1}^n 1-x_ig_i \right)}\\
    &= \prod_{i=1}^n\frac{\det ((1-x_i^{m_i}) I_r)}{\det(I_r - x_i\phi(g_i))}\\
    &= \prod_{i=1}^n\frac{(1-x_i^{m_i})^r}{\theta_{\phi(g_i)}(x_i)}.
  \end{align*}

\end{proof}

\begin{corollary}\label{cor:reg_basis}
  Suppose $G$ has a perfect basis $(g_1,\ldots,g_n)$, and set
  \[ \alpha = \sum_{\substack{g=g_1^{c_1}\cdots g_n^{c_n}\\0\leq c_i<m_i}} q^{c_1+\cdots+c_n}\cdot g. \]
  If $\phi_{\reg}$ is the regular representation of $G$, then
  \[ \Delta_{\phi_{\reg}}(\alpha)=\prod_{i=1}^n(1-q^{m_i})^{|G|(1-1/o(g_i))}. \]
\end{corollary}

\begin{proof}
  Specializing $q=x_i$ in Theorem~\ref{thm:general_basis} gives
  \[ \Delta_{\phi_{\reg}}(\alpha)=\prod_{i=1}^n\frac{(1-q^{m_i})^{|G|}}{\theta_{\phi_{\reg}(g_i)}(q)}. \]
  But $\theta_{\phi_{\reg}(g_i)}(q)=(1-q^{m_i})^{|G|/o(g_i)}$ by Corollary~\ref{cor:det_reg}.
\end{proof}

\begin{example}\label{ex:Dn_reps}
  Let $G$ be the dihedral group in Example~\ref{ex:Dn_basis}. For $h\in G$, set $\rot(h)=h(n)$ if $h(n)\neq n$, and $\rot(h)=0$ if $h(n)=n$. Let $\refl(h)=0$ if $h$ is a rotation and $\refl(h)=1$ if $h$ is a reflection. Then for $0\leq c_1<n,\ 0\leq c_2<2$, we have $\rot(g_1^{c_1}g_2^{c_2})=c_1$ and $\refl(g_1^{c_1}g_2^{c_2})=c_2$.

  Let $\alpha=\sum_{h\in G}x_1^{\rot(h)}x_2^{\refl(h)}\cdot h$, and let $\phi$ be a representation of $G$ of dimension $r$. By Theorem~\ref{thm:general_basis}, we have
  \[ \Delta_{\phi}(\alpha) = \frac{(1-x_1^n)^r}{\theta_{\phi(g_1)}(x_1)}\frac{(1-x_2^2)^r}{\theta_{\phi(g_2)}(x_2)}. \]
  If $\phi=\phi_{\reg}$ is the regular representation, then $r=2n,\ \theta_{\phi(g_1)}(q)=(1-q^n)^2$, and $\theta_{\phi(g_2)}(q)=(1-q^2)^n$. Hence,
  \[ \Delta_{\phi_{\reg}}(\alpha)=(1-x_1^n)^{2n-2}(1-x_2^2)^n. \]
  
  We end this example by evaluating $\Delta_{\phi}(\alpha)$ for all irreducible representations over $\Qbb$. Let $C_n = \langle g_1 \rangle$.
\begin{itemize}
\item
For the trivial representation $\phi_{\triv}$ we have
\[
\Delta_{\phi_{\triv}}(\alpha)
=
\frac{(1-x_1^n) (1-x_2^2)}{(1-x_1)(1-x_2)}.
\]
\item
There is a one-dimensional representation $\phi_{\sign}(g) = (-1)^{c_2(g)}$ coming from the action on the cosets of $C_n$.
We have
\[
\Delta_{\phi_{\sign}}(\alpha)
=
\frac{(1-x_1^n)(1-x_2^2)}{(1-x_1)(1+x_2)}.
\]
\item
For each divisor $d$ of $n$, there is an irreducible representation $\rho_d$ of $C_n$. 
Writing
\[
\Phi_d(q) = a_0 + a_1 q + \cdots + q^{\ell}
\]
for the $d$th cyclotomic polynomial, we can realize $\rho_d$ on $V = \Qbb^{\ell}$ via
\[
\rho_d(g_1)
=
\begin{bmatrix}
0 & 0 & 0 & \cdots & 0 & 0 & -a_0 \\
1 & 0 & 0 & \cdots & 0 & 0 & -a_1 \\
0 & 1 & 0 & \cdots & 0 & 0 & -a_2 \\
0 & 0 & 1 & \cdots & 0 & 0 & -a_3 \\
\vdots & \vdots & \vdots & \ddots & \vdots & \vdots & \vdots \\
0 & 0 & 0 & \cdots & 1 & 0 & -a_{\ell-2} \\
0 & 0 & 0 & \cdots & 0 & 1 & -a_{\ell-1} \\
\end{bmatrix}
\]
Let $A$ be the anti-diagonal matrix with all anti-diagonal entries equal to $1$.
Then $A^2$ is the identity and  $A \rho_d(g_1) = \rho_d(g_1^{-1}) A$. Hence, we have an irreducible representation $\phi_d$ of the dihedral group via $\phi_d(h) = \rho_d(g_1)^{\rot(h)} A^{\refl(h)}$.
In this representation, one has $\theta_{\phi_d(g_1)}(q)=q^{\ell}\Phi_d(1/q)$ and
\[
\theta_{\phi_d(g_2)}(q) = 
\begin{cases}
(1-q^2)^{\frac{\ell}{2}} & \ell \textup{ even} \\
(1-q^2)^{\frac{\ell-1}{2}} (1-q) & \ell \textup{ odd}
\end{cases}
\]
from which $\Delta_{\phi_d}(\alpha)$ can be written down.
\item
Lastly, for any divisor $d$ of $n$ we can tensor the representation $\phi_d$ with the representation $\psi$ to obtain another irreducible representation.
\end{itemize}

\end{example}

\section{Major index matrix}\label{sec:maj}

For $k\in[n]$, let $t_k=(k,k-1,\ldots,1)$ be a $k$-cycle in $\Sfrak_n$. In \cite{adin2001flag}, Adin and Roichman gave an alternative interpretation of the major index of a permutation. We recall this statement and its proof.

\begin{lemma}[Claim 2.1, \cite{adin2001flag}]\label{lem:maj_basis}
  The $(n-1)$--tuple $(t_n,t_{n-1},\ldots,t_2)$ is a perfect basis of $\Sfrak_n$. Moreover, if $w=t_n^{c_n}t_{n-1}^{c_{n-1}}\cdots t_2^{c_2}$ for some $c_i$ with $0\le c_i< i$, then $\maj(w)=c_n+c_{n-1}+\cdots+c_2$.
\end{lemma}

\begin{proof}
  Let $w\in\Sfrak_n$ be given. Since $t_n$ is the $n$--cycle $(n,n-1,\ldots,1)$, there exists a unique $c_n$ with $0\le c_n < n$ such that $t_n^{-c_n}w$ fixes $n$. By similar reasoning, there is a unique $c_{n-1}$ with $0\le c_{n-1}< n-1$ such that $t_{n-1}^{-c_{n-1}}t_n^{-c_n}w$ fixes $n-1$. Since $t_{n-1}$ fixes $n$, the element $t_{n-1}^{-c_{n-1}}t_n^{-c_n}w$ also fixes $n$. Continuing in this manner, we find unique $c_i$ with $0\le c_i< i$ such that $t_2^{-c_2}\cdots t_n^{-c_n}w$ is the identity permutation. Hence, the factorization $w=t_n^{c_n}\cdots t_2^{c_2}$ is unique.

  As there are $n!=|\Sfrak_n|$ choices for the exponents $(c_n,\ldots,c_2)$, we conclude that $(t_n,\ldots,t_2)$ is a basis. In fact, it is a perfect basis since the order of $t_i$ is $i$ and the exponent $c_i$ can be any value in the range $0\leq c_i<i$.

  For each $k$, let $g_k=t_n^{c_n}\cdots t_k^{c_k}$. We prove that the major index of $g_k$ is $c_n+\cdots+c_k$. Taking $k=2$ gives $\maj(w)=c_n+\cdots+c_2$.
  
  We observe that the values in $g_n$ are \emph{cyclically ordered}, i.e. there exists a unique $j\in\Zbb/n\Zbb$ such that $g_n(|j|+1)<\cdots<g_n(n)<g_n(1)<\cdots<g_n(|j|)$. Moreover, $|j|=c_n$, so the major index of $g_n$ is $c_n$.

  Now let $k\geq 2$ and suppose the first $k+1$ values of $g_{k+1}$ are cyclically ordered. Multiplying $g_{k+1}$ on the right by $t_k$ rotates the first $k$ values of $g_{k+1}$. Hence, the first $k$ values of $g_k$ are cyclically ordered.

  If the first $k$ values of $g_{k+1}$ are in increasing order, then by the same argument as in the base case, $\maj(g_k)=\maj(g_{k+1}t_k^{c_k})=\maj(g_{k+1})+c_k$.

  Otherwise, there exists $j\in\Zbb/k\Zbb,\ j\neq 0$ such that $g_{k+1}(|j|+1)<\cdots<g_{k+1}(k+1)<g_{k+1}(1)<\cdots<g_{k+1}(|j|)$. Then $g_k$ has a descent at $|j|+c_k$ if $|j|+c_k\leq k$, or $g_k$ has descents at $|j|+c_k-k$ and at $k$ if $k+1\le |j|+c_k\le k+|j|-1$. All higher descents of $g_k$ are shared with $g_{k+1}$. Hence, $\maj(g_k)=\maj(g_{k+1})+c_k$, as desired.
\end{proof}

For the remainder of the section, we consider the element $\alpha\in \Cbb(q)\Sfrak_n$ where
\[ \alpha = \sum_{w\in\Sfrak_n} q^{\maj(w)}\cdot w. \]

Lemma~\ref{lem:maj_basis} together with Theorem~\ref{thm:general_basis} immediately implies the following.

\begin{corollary}\label{cor:maj_general}
  Let $n\geq 1$ be given. If $V\cong\Cbb^r$ and $\phi:\Sfrak_n\ra\GL(V)$, then
  \[ \Delta_{\phi}(\alpha) = \prod_{k=2}^n\frac{(1-q^k)^r}{\theta_{\phi(t_k)}(q)}. \]
\end{corollary}

If $\phi_{\reg}$ is the regular representation of $\Sfrak_n$, then
\[ \phi_{\reg}\left(\sum q^{\maj(w)}\cdot w\right) = \left(q^{\maj(uv^{-1})}\right)_{u,v\in\Sfrak_n}. \]
Theorem~\ref{thm:main_maj} now follows immediately from Corollary~\ref{cor:reg_basis}.

\begin{example}
\label{eg:defining}
The symmetric group $\Sfrak_n$ naturally acts on $[n]$. The corresponding permutation representation $\phi_{\DEF}$ is called the \emph{defining representation}. Explicitly, the matrix representing $\alpha$ is
\[
\phi_{\DEF}(\alpha) = \left(\sum_{w(i)=j}q^{\maj(w)}\right)_{i,j\in[n]}.
\]
We verify that
\[
\det\left(\sum_{\substack{w\in\Sfrak_n\\w(i)=j}}q^{\maj(w)}\right)_{i,j\in[n]} = (1-q)^{\binom{n}{2}}([n]!_q)^{n-1}.
\]
The left-hand side is $\Delta_{\phi_{\DEF}}(\alpha)$, so
\[
\det\left(\sum_{\substack{w\in\Sfrak_n\\w(i)=j}}q^{\maj(w)}\right)_{i,j\in[n]} = \prod_{k=2}^{n}\frac{(1-q^k)^n}{\theta_{\phi_{\DEF}(t_k)}(q)}.
\]
The element $t_k$ has one orbit of size $k$ and $n-k$ orbits of size $1$. Hence,
\[
\theta_{\phi_{\DEF}(t_k)}(q) = (1-q^k)(1-q)^{n-k}.
\]
Hence, the determinant is equal to 
\begin{align*}
    \prod_{k=2}^n\frac{(1-q^k)^n}{\theta_{\phi_{\DEF}(t_k)}(q)}
    &= \prod_{k=2}^{n}\frac{(1-q^k)^n}{(1-q^k)(1-q)^{n-k}}\\
    &= \prod_{k=2}^{n}([k]_q)^{n-1}(1-q)^{k-1} = (1-q)^{\binom{n}{2}}([n]!_q)^{n-1}.
\end{align*}
\end{example}

\begin{example}
One can consider the action of $\Sfrak_n$ on tuples.
For example, here we calculate $\Delta_{\phi}(\alpha)$ where $\phi$ is the representation determined by the action of $\Sfrak_n$ on $\binom{n}{2}$.
From Corollary~\ref{cor:det_reg} and Theorem~\ref{thm:general_basis} it suffices to determine the orbit decomposition of the cycles $t_k$ acting on $\binom{n}{2}$.

For the $t_k$ orbit of $(i,j) \in \binom{n}{2}$ with $i < j$ there are three possibilities:
\begin{itemize}
\item
$k < i$ in which case $t_k$ fixes $(i,j)$;
\item
$i \le k < j$ in which case the $j$ is fixed by $t_k$ and the orbit has size $k$;
\item
$j \le k$ in which case the orbit is the same as the orbit of $(i,j) \in \binom{k}{2}$.
\end{itemize}
It therefore suffices to determine the orbit structure of the action of $t_k$ on $\binom{k}{2}$.
When $k$ is odd all orbits have size $k$.
When $k$ is even the possibility $2(j-i) = k$ gives the unique orbit of size $\frac{k}{2}$.
We therefore have, for the three possibilities above:
\begin{itemize}
\item
$\binom{n-k}{2}$ orbits of size 1;
\item
$n-k$ orbits of size $k$;
\item
$\frac{k-1}{2}$ orbits of size $k$ when $k$ is odd \textsc{or} $\frac{k-2}{2}$ orbits of size $k$ and one orbit of size $\frac{k}{2}$ when $k$ is even;
\end{itemize}
and can calculate that
\[
\theta_{\phi(t_k)}(q)
=
\begin{cases}
(1-q)^{\binom{n-k}{2}} (1-q^k)^{n-k} (1-q^k)^{\frac{k-1}{2}} & k \textup{ odd} \\
(1-q)^{\binom{n-k}{2}} (1-q^k)^{n-k} (1-q^k)^{\frac{k-2}{2}} (1-q^{\frac{k}{2}})  & k \textup{ even}
\end{cases}
\]
with the determinant formula following from Theorem~\ref{thm:general_basis}.
\end{example}

We are interested in the extent to which $\Delta_{\phi_\lambda}(\alpha) = \det \phi_\lambda(\alpha)$ can be calculated where $\lambda$ is any partition of $n \in \N$ and $\phi_\lambda$ is the corresponding irreducible representation of $\Sfrak_n$.
The factorization 
\[
\prod_{i=2}^n (1+q t_i + \cdots + q^{i-1}t_i^{i-1}) (1-q t_i) = \prod_{i=2}^n(1 - q^i)\cdot 1
\]
from $x_i = q$ and $g_i = t_i$ in \eqref{eqn:alphaBetaFactor}, together with the fact that $t_{n-1},\dots,t_2$ is a perfect basis of $\Sfrak_{n-1}$, suggests an inductive approach.
Indeed, if $\lambda$ is a partition of $n$, the restriction $\phi_\lambda | \Sfrak_{n-1}$ of $\phi_\lambda$ to $\Sfrak_{n-1}$ is known to be a direct sum
\[
\phi_\lambda | \Sfrak_{n-1} = \bigoplus_{\eta \prec \lambda} \phi_\eta
\]
of those $\phi_\eta$ where $\eta$ immediately precedes $\lambda$ in the Young lattice.
Thus, for $2 \le i \le n-1$ we have
\begin{equation}
\label{eqn:irredInduction}
\det \phi_\lambda(1-qt_i) = \prod_{\eta \prec \lambda} \det \phi_\eta(1-qt_i)
\end{equation}
and it remains to calculate $\theta_{\phi_\lambda(t_n)}(q) = \det \phi_\lambda(1 - qt_n)$. For this calculation it suffices to determine the eigenvalues of $\phi_\lambda(t_n)$. These can be found using work of Stembridge~\cite[Theorem~3.3]{stembridge1989eigenvalues} which we recall here.

Fix a partition $\lambda$ of $n$ and $g \in \Sfrak_n$ of order $m$.
The eigenvalues of $\phi_\lambda(g)$ are of the form $\omega^{e_1},\dots,\omega^{e_r}$ where $\omega = e^{2 \pi i / m}$.
The exponents $e_1,\dots,e_r$ are called the cyclic exponents of $g$ and are defined modulo $m$.

A standard tableau over $\lambda$ is any filling of $\lambda$ by $\{1,\dots,n\}$ with rows and columns strictly increasing.
One calls $1 \le k \le n$ a descent of a standard tableau if $k+1$ appears in a row strictly below that of $k$.

Let $\mu = (\mu_1,\mu_2,\dots,\mu_\ell)$ be the cycle type of our element $g \in \Sfrak_n$.
Form
\[
b_\mu = \left( \frac{m}{\mu_1}, \frac{2m}{\mu_1}, \dots, m, \frac{m}{\mu_2},\frac{2m}{\mu_2},\dots, m,\dots \right)
\]
which is a tuple of length $\mu_1 + \cdots + \mu_\ell$.
For example
\[
b_{(4,4,3,2)} = (3,6,9,12,3,6,9,12,4,8,12,6,12)
\]
and if $g$ is an $n$-cycle we have $b_{(n)} = (1,2,\dots,n)$.
For any standard tableau $T$ over $\lambda$ its $\mu$ index is
\[
\ind_\mu(T) = \sum_{k \in D(T)} b_\mu(k) \bmod m
\]
where $D(T)$ is the set of descents of $T$.
The content of \cite[Theorem~3.3]{stembridge1989eigenvalues} is that
\[
q^{e_1} + \cdots + q^{e_r} = \sum_{T | \lambda} q^{\ind_\mu(T)}
\]
modulo $1-q^m$.

\begin{example}[The standard representation]
The standard representation corresponds to the partition $\lambda = [n-1,1]$. We will calculate the eigenvalues of $\phi_\lambda(t_n)$. The standard tableaux over $\lambda$ are indexed by the entry $2,\dots,n$ on the second row. Each has a single descent of $1,\dots,n-1$ respectively. We conclude that $\phi_\lambda(t_n)$ has eigenvalues $\omega,\omega^2,\dots,\omega^{n-1}$ and that its characteristic polynomial is $[n]_q$.
Then
\[
\theta_{\phi_\lambda(t_n)}(q) = q^{n-1} \det \phi_\lambda(q - t_n) = [n]_q
\]
as well.
For all $2 \le i \le n-1$ we have
\begin{equation}
\label{eqn:standardInduction}
\theta_{\phi_\lambda(t_i)}(q) = (1-q)^{n-i} \theta_{\phi_{[i-1,1]}(t_i)}(q)
=
(1-q)^{n-i} [i]_q
\end{equation}
from repeated application of \eqref{eqn:irredInduction}.
We conclude that
\[
\Delta_\lambda(\alpha)
=
\dfrac{\displaystyle\prod_{i=2}^n (1-q^i)^{n-1}}{\displaystyle\prod_{i=2}^n (1-q)^{n-i} [i]_q}
=
\frac{1}{[n]_q!} \dfrac{\displaystyle \prod_{i=2}^n (1-q^i)^{n-1}}{\displaystyle \prod_{i=2}^n(1-q)^{n-1}} \prod_{i=2}^n (1-q)^{i-1}
=
([n]_q!)^{n-2} (1-q)^{\binom{n}{2}}
\]
which can also be obtained from dividing the result of Example~\ref{eg:defining} by $[n]_q!$.
\end{example}

\begin{example}[The {$[2,2]$} representation]
Fix $\lambda = [2,2]$. The two standard tableaux over $\lambda$ are
\[
T = \young(12,34)
\qquad
S = \young(13,24)
\]
with descent sets $\{2\}$ and $\{1,3\}$ respectively.
The element $t_4$ has eigenvalues $-1$ and $1$ so its characteristic polynomial is $q^2 - 1$. Thus
\[
\theta_{\phi_\lambda(t_4)}(q)
=
q^2 \det \phi_\lambda(\tfrac{1}{q} - t_4)
=
1-q^2
\]
and
\begin{align*}
\theta_{\phi_\lambda(t_3)}(q)
&
=
\theta_{\phi_{[2,1]}(t_3)}(q)
=
[3]_q
\\
\theta_{\phi_\lambda(t_2)}(q)
&
=
\theta_{\phi_{[2,1]}(t_2)}(q)
=
(1-q) [2]_q
\end{align*}
from the previous example.
Finally
\[
\Delta_\lambda(\alpha)
=
\dfrac{(1-q^2)^2 (1-q^3)^2 (1-q^4)^2}{(1-q^2) \cdot (1+q+q^2) \cdot (1-q) (1 + q)}
=
(1-q)(1-q^3)(1-q^4)^2
\]
\end{example}

\section{Flag major index matrix}\label{sec:fmaj}

Let $H,N$ be groups such that $H$ acts on $N$ on the right. The \emph{semidirect product} $H\ltimes N$ is the group whose elements are $(g,x)$ for $g\in H,\ x\in N$, where
\[ (g,x)(h,y) = (gh,(x\cdot h)y). \]

The symmetric group $\Sfrak_n$ acts on $(\Zbb/m\Zbb)^n$ by permuting coordinates. That is, if $w\in\Sfrak_n$ and $x\in(\Zbb/m\Zbb)^n$, then $x\cdot w\in(\Zbb/m\Zbb)^n$ where $(x\cdot w)_i = x_{w(i)}$. The group of colored permutations is the semidirect product $\Sfrak_n^m=\Sfrak_n\ltimes (\Zbb/m\Zbb)^n$. 
We express a colored permutation $(w,x)$ by writing $w$ in one--line notation with $x_k$ bars above $w(k)$. For example, the colored permutation $(1342,\ (1,0,2,1))$ is written $\bar{1}3\bar{\bar{4}}\bar{2}$.

Let $b=(1,0,0,\ldots,0)\in (\Zbb/m\Zbb)^n$. As in Section~\ref{sec:maj}, we set $t_k=(k,k-1,\ldots,1)\in\Sfrak_n$. In particular, we let $t_1$ be the identity permutation.

For $k\in[n]$, let $\tilde{t}_k=(t_k,b)\in\Sfrak_n^m$. Then the order of $\tilde{t}_k$ is $mk$. For example, if $m=n=k=3$, then
\[ \langle \tilde{t}_k \rangle = \{ 123,\ \bar{3}12,\ \bar{2}\bar{3}1,\ \bar{1}\bar{2}\bar{3},\ \bar{\bar{3}}\bar{1}\bar{2},\ \bar{\bar{2}}\bar{\bar{3}}\bar{1},\ \bar{\bar{1}}\bar{\bar{2}}\bar{\bar{3}},\ 3\bar{\bar{1}}\bar{\bar{2}},\ 23\bar{\bar{1}} \}. \]

Colored letters are totally ordered as
\[
n>(n-1)>\cdots>1>\bar{n}>\cdots>\bar{1}>\bar{\bar{n}}>\cdots
\]
giving rise to a major index for colored permutations.
For example, the colored permutation $\bar{1}3\bar{\bar{4}}\bar{2}$ only has a descent at $2$ since $3>\bar{\bar{4}}$ but $\bar{1}<3$ and $\bar{\bar{4}}<\bar{2}$. So, the major index is $\maj(\bar{1}3\bar{\bar{4}}\bar{2})=2$.

The \emph{flag major index} of a colored permutation $g\in\Sfrak_n^m$ is $\fmaj(g)=m\maj(g)+\col(g)$. For example, if $m=3$, then $\fmaj(\bar{1}3\bar{\bar{4}}\bar{2})=3\cdot 2 + 4 = 10$. This statistic was introduced by Adin and Roichman in \cite{adin2001flag} to give a combinatorial formula for the Hilbert series of a certain ring of invariants.

The proof of the following lemma is similar to the symmetric group case, and will be omitted. It can be obtained from \cite[Proposition 2.1]{Shwartz2007MajorIA} and \cite[Theorem 3.1]{adin2001flag}.

\begin{lemma}\label{lem:fmaj_basis}
  The $n$--tuple $(\tilde{t}_n,\ldots,\tilde{t}_1)$ is a perfect basis of $\Sfrak_n^m$. Moreover, if $g=\tilde{t}_n^{c_n}\cdots\tilde{t}_1^{c_1}$ for some $0\le c_k < mk$, then $\fmaj(g)=c_n+\cdots+c_1$.
\end{lemma}

\begin{example}
  Consider the colored permutation $g=\bar{1}3\bar{\bar{4}}\bar{2}$. The only power of $\tilde{t}_4$ with $\bar{2}$ in the last position is $(\tilde{t}_4)^6=\bar{\bar{3}}\bar{\bar{4}}\bar{1}\bar{2}$. This permutation has no descents, so its flag major index is
  \[ \fmaj(\tilde{t}_4^6)=3\maj(\tilde{t}_4^6)+\col(\tilde{t}_4^6)=0+6=6. \]
  Next, we rotate the first three entries once to put $\bar{\bar{4}}$ into the third position, i.e. $\tilde{t}_4^6\tilde{t}_3=\bar{\bar{1}}\bar{\bar{3}}\bar{\bar{4}}\bar{2}$. There are still no descents, and its flag major index is $7$.
  Rotating the first two entries twice will put $3$ into the second position, i.e. $\tilde{t}_4^6\tilde{t}_3\tilde{t}_2^2=13\bar{\bar{4}}\bar{2}$. The colors are removed from the first two values, but a descent at $2$ is created, so
  \[ \fmaj(\tilde{t}_4^6\tilde{t}_3\tilde{t}_2^2)=3\cdot 2+3=9. \]
  Finally, we change the color of the first entry to find $g=\tilde{t}_4^6\tilde{t}_3\tilde{t}_2^2\tilde{t}_1$ and $\fmaj(g)=10$ is the sum of the exponents of this factorization.
\end{example}

\begin{theorem}\label{thm:fmaj}
  Let $n,m\geq 1$.
  \[ \det\left(q^{\fmaj(gh^{-1})}\right)_{g,h\in\Sfrak_n^m} = \prod_{k=1}^n(1-q^{mk})^{n!m^n(1-1/(mk))} \]  
\end{theorem}

\begin{proof}
  Set $\alpha=\sum q^{\fmaj(g)}\cdot g$. If $\phi_{\reg}$ is the regular representation of $\Sfrak_n^m$, then
\[ \phi_{\reg}(\alpha) = \left(q^{\fmaj(gh^{-1})}\right)_{g,h\in\Sfrak_n^m}. \]

By Lemma~\ref{lem:fmaj_basis} and Corollary~\ref{cor:reg_basis},
\begin{align*}
  \Delta_{\phi_{\reg}}(\alpha) &= \prod_{k=1}^n(1-q^{o(\tilde{t}_k)})^{|\Sfrak_n^m|(1-1/o(\tilde{t}_k))}\\
  &= \prod_{k=1}^n(1-q^{mk})^{n!m^n(1-1/(mk))}\qedhere
\end{align*}

\end{proof}

We identify $\Sfrak_n$ with the subgroup of colored permutations $\{(\mathbf{0},w)\in\Sfrak_n^m \mid w\in\Sfrak_n\}$. Observe that for $g\in\Sfrak_n^m$, we have $g\in\Sfrak_n$ if and only if $\col(g)=0$. If $h$ is any colored permutation, there is a unique ordering of the colored values of $h$ with no descents. That is, there is a unique $g\in\Sfrak_n$ such that $\maj(hg)=0$. Hence, the set $T=\{h\in\Sfrak_n^m \mid \maj(h)=0\}$ is a left transversal to $\Sfrak_n$ in $\Sfrak_n^m$.

\begin{lemma}\label{lem:fmaj_split}
  Let $T$ be the transversal to $\Sfrak_n$ in $\Sfrak_n^m$ defined above. Then
  \[ \sum p^{\maj(g)}q^{\col(g)}\cdot g = \left(\sum_{h\in T} q^{\col(h)}\cdot h\right)\left(\sum_{w\in\Sfrak_n} p^{\maj(w)}\cdot w\right). \]
\end{lemma}

\begin{proof}
  We first expand the right--hand side of the equation. Then
  \begin{align*}
    \left(\sum_{h\in T} q^{\col(h)}\cdot h\right)\left(\sum_{w\in\Sfrak_n} p^{\maj(w)}\cdot w\right) &= \sum_{g\in\Sfrak_n^m} p^{\maj(w)}q^{\col(h)}\cdot g,
  \end{align*}
  where in the latter sum, $g=hw,\ h\in T$, and $w\in\Sfrak_n$. Fix $g\in\Sfrak_n^m$, and decompose $g=hw$ accordingly. Since multiplication by $w$ on the right rearranges colors without changing their values, it is clear that $\col(g)=\col(h)$. On the other hand, since the colored values of $h$ are in increasing order, it follows that $w$ and $g$ have the same descents. Hence, $\maj(g)=\maj(w)$. We conclude that $p^{\maj(w)}q^{\col(h)}=p^{\maj(g)}q^{\col(g)}$, as desired.
\end{proof}

Let $H$ be a subgroup of $G$. For $g\in G$, the subgroup $H$ acts on the right as the regular representation on the vector space $\Qbb[gH]$. Hence, the restriction of the regular representation of $G$ is isomorphic to a direct sum of $[G:H]$ copies of the regular representation of $H$.

\begin{theorem}
  For $m,n\geq 1$,
  \[ \det\left(p^{\maj(gh^{-1})}q^{\col(gh^{-1})}\right)_{g,h\in\Sfrak_n^m} = \prod_{k=2}^n(1-p^k)^{n!m^n(k-1)/k}\prod_{k=1}^n(1-q^{mk})^{n!m^{n-1}(m-1)/k}. \] 
\end{theorem}

\begin{proof}
  Let $\alpha=\sum q^{\fmaj(g)}\cdot g$ as in the proof of Theorem~\ref{thm:fmaj}. Let $\beta=\sum p^{\maj(g)}q^{\col(g)}\cdot g$ in $\Qbb(p,q)\Sfrak_n^m$. Then
  \[ \phi_{\reg}(\beta) = \left(p^{\maj(gh^{-1})}q^{\col(gh^{-1})}\right)_{g,h\in\Sfrak_n^m}, \]
  so we seek to prove that the right--hand side of the theorem statement is equal to $\Delta_{\phi_{\reg}}(\beta)$. By Lemma~\ref{lem:fmaj_split}, we have $\beta = \left(\sum_{h\in T} q^{\col(h)}\cdot h\right)\left(\sum_{w\in\Sfrak_n} p^{\maj(w)}\cdot w\right)$. Hence, there exists polynomials $A(p),B(q)$ such that $\Delta_{\phi_{\reg}}(\beta)=A(p)B(q)$. Since $\alpha=\beta\mid_{p=q^m}$, we have $\Delta_{\phi_{\reg}}(\alpha)=A(q^m)B(q)$, so
  \[ A(q^m)B(q) = \prod_{k=1}^n(1-q^{mk})^{n!m^n(1-1/(mk))}. \]

  Since $\Sfrak_n$ is a subgroup of $\Sfrak_n^m$ of index $m^n$, the restriction of the regular representation of $\Sfrak_n^m$ to $\Sfrak_n$ is isomorphic to a direct sum of $m^n$ copies of the regular representation of $\Sfrak_n$. Combined with Theorem~\ref{thm:main_maj}, we have
  \[ A(p) = \Delta_{\phi_{\reg}}\left(\sum_{w\in\Sfrak_n} p^{\maj(w)}\cdot w\right) = \prod_{k=2}^n (1-p^k)^{n!m^n (k-1)/k}. \]
  Therefore,
  \begin{align*}
    B(q) &= \frac{\Delta_{\phi_{\reg}}(\beta)}{A(q^m)}\\
    &= \frac{\prod_{k=1}^n(1-q^{mk})^{n!m^n(1-1/(mk))}}{\prod_{k=2}^n (1-q^{mk})^{n!m^n(k-1)/k}}\\
    &= \prod_{k=1}^n(1-q^{mk})^{n!m^{n-1}(m-1)/k}.
  \end{align*}
  The theorem now follows by multiplying the formulas for $A(p)$ and $B(q)$.
\end{proof}

\section{Absolute flag major index matrix}\label{sec:amaj}

In contrast with Section~\ref{sec:fmaj}, here we consider a simpler statistic that takes the descents of $(w,x)$ to be those of $w$. The \emph{absolute major index} of $(w,x)$ is $\amaj(w,x)=\maj(w)$. The \emph{absolute flag major index} of $(w,x)$ is $\amaj(w,x)+\col(w,x)$.
We prove the following identity. Krattenthaler conjectured the $m=2$ case in \cite[Conjecture 48]{krattenthaler2005advanced}.

\begin{theorem}\label{thm:amaj}
  For all $m,n\geq 1$, we have
  \[ \det\left(p^{\amaj(gh^{-1})}q^{\col(gh^{-1})}\right)_{g,h\in\Sfrak_n^m} = (1-q^m)^{n!m^{n-1}(m-1)n}\prod_{k=2}^n(1-p^k)^{n!m^n(k-1)/k}. \]
\end{theorem}

To prove Theorem~\ref{thm:amaj}, we produce a different perfect basis of $\Sfrak_n^m$ than the one considered in Section~\ref{sec:fmaj}. The construction of this perfect basis can be formulated more generally as follows.

Let $H,N$ be groups such that $H$ acts on $N$ on the right. Consider the semidirect product $G=H\ltimes N$. We identify $H$ and $N$ with the subgroups $\{(h,1)\in G\mid h\in H\}$ and $\{(1,x)\in G\mid x\in N\}$, respectively. If $(h_1,\ldots,h_k)$ is a perfect basis of $H$ and $(x_1,\ldots,x_{\ell})$ is a perfect basis of $N$, then $(h_1,\ldots,h_k,x_1,\ldots,x_{\ell})$ is a perfect basis of $G$.

For $\Sfrak_n^m$ we combine our perfect basis for $\Sfrak_n$ with one for $\Zbb/m\Zbb^n$.
For each $i\in[n]$, let $y^{(i)}\in(\Zbb/m\Zbb)^n$ where
\[ (y^{(i)})_j=\begin{cases}1\ &\mbox{if }i=j\\0\ &\mbox{else}\end{cases}. \]
It is clear that $(y^{(1)},\ldots,y^{(n)})$ is a perfect basis of $(\Zbb/m\Zbb)^n$ since $x=\sum_i |x_i|y^{(i)}$ for all $x\in(\Zbb/m\Zbb)^n$. Hence, $(t_n,\ldots,t_2,y^{(1)},\ldots,y^{(n)})$ is a perfect basis of $\Sfrak_n^m$. Moreover, if $g=t_n^{c_n}\cdots t_2^{c_2}(y^{(1)})^{d_1}\cdots(y^{(n)})^{d_n}$ is the factorization of $g$, then $\amaj(g)=c_n+\cdots+c_2$ and $\col(g)=d_1+\cdots+d_n$.

\begin{proof}[Proof of Theorem~\ref{thm:amaj}]
Let $\phi_{\reg}$ be the regular representation of $\Sfrak_n^m$. The restriction of $\phi_{\reg}$ to $\Sfrak_n$ is isomorphic to a direct sum of $[\Sfrak_n^m:\Sfrak_n]=m^n$ copies of the regular representation of $\Sfrak_n$. The restriction to $(\Zbb/m\Zbb)^n$ is isomorphic to a direct sum of $[\Sfrak_n^m:(\Zbb/m\Zbb)^n]=n!$ copies of the regular representation of $(\Zbb/m\Zbb)^n$. We deduce the following sequence of identities.

\begin{align*}
  \det\left(p^{\amaj(gh^{-1})}q^{\col(gh^{-1})}\right)_{g,h\in\Sfrak_n^m} &= \Delta_{\phi_{\reg}}\left(\sum_{g\in\Sfrak_n^m} p^{\amaj(g)}q^{\col(g)}\cdot g\right)\\
  &= \Delta_{\phi_{\reg}}\left(\sum_{w\in\Sfrak_n}p^{\maj(w)}\cdot w\right)\Delta_{\phi_{\reg}}\left(\sum_{x\in(\Zbb/m\Zbb)^n} q^{\col(x)}\cdot x\right)\\
  &= \prod_{k=2}^n(1-p^k)^{n!m^n (k-1)/k}\prod_{i=1}^n(1-q^m)^{n!m^{n-1}(m-1)}\qedhere
\end{align*}
\end{proof}


\section{Signed permutations}\label{sec:signed}

A \emph{signed permutation} is a pair $(\veps,w)$ where $w\in\Sfrak_n$ and $\veps\in\{-1,1\}^n$. We refer to $\veps$ as the sign vector of the signed permutation $(\veps,w)$. The symmetric group acts on the set of sign vectors on the left such that for $w\in\Sfrak_n,\ \veps\in\{-1,1\}^n$, $(w\cdot\veps)_i = \veps_{w^{-1}(i)}$ for all $i$. Let $B_n$ be the group of signed permutations, i.e. the semidirect product $\{1,-1\}^n\rtimes\Sfrak_n$ where $(\veps,u)(\veps^{\pr},v)=(\veps(u\cdot\veps^{\pr}),uv)$. This is also known in the literature as the hyperoctahedral group since it is isomorphic to the group of symmetries of a hyperoctahedron.

We may write a signed permutation in one--line notation with a bar above a value $i$ if $\veps_i=-1$. For example, the signed permutation
\[ \left((1, -1, -1, 1),\ \binom{1\ 2\ 3\ 4}{2\ 1\ 4\ 3}\right) \]
would be written as $\bar{2}14\bar{3}$. We refer to elements of $\{1,2,3,\ldots,\bar{1},\bar{2},\bar{3},\ldots\}$ as signed letters.

We consider two total orderings on signed letters. The first ordering is the natural ordering on integers,
\[ \cdots <_A \bar{n} <_A \cdots <_A \bar{1} <_A 0 <_A 1 <_A \cdots  <_A n <_A \cdots. \]
The second ordering is
\[ 0 <_B 1 <_B \cdots <_B n <_B \cdots <_B \bar{n} <_B \cdots <_B \bar{2} <_B \bar{1}. \]

For $i\in[n-1]$, we say $i$ is \emph{A--descent} of a signed permutation $(\veps,w)=w_1\cdots w_n$ if $w_i >_A w_{i+1}$. Furthermore, $0$ is an A--descent if $w_1$ is negative. Similarly, $i$ is a \emph{B--descent} of $w=w_1\cdots w_n$ if $w_i >_B w_{i+1}$. Furthermore, $n$ is a B--descent if $w_n$ is negative. Let $\maj_A(\veps,w)$ (respectively, $\maj_B(\veps,w)$) be the sum of the A--descents (respectively, B--descents) of $(\veps,w)$. 

The \emph{negative set} is $\Neg(\veps,w)=\{i \mid \veps_i=-1\}$. Let $\nneg(\veps,w)=|\Neg(\veps,w)|$, and let $\sneg(\veps,w)$ be the sum of elements in $\Neg(\veps,w)$.

For example, $\{0,3\}$ is the set of A--descents of $\bar{2}14\bar{3}$, so $\maj_A(\bar{2}14\bar{3})=3$. The set of B--descents of $\bar{2}14\bar{3}$ is $\{1,4\}$, so $\maj_B(\bar{2}14\bar{3})=5$. The negative set is $\Neg(\bar{2}14\bar{3})=\{2,3\}$, so $\nneg(\veps,w)=2$ and $\sneg(\veps,w)=5$.

The statistic $\maj_B$ was introduced by Reiner in \cite{reiner1993signed}. The statistics $\maj_A$ and $\sneg$ were used by Adin, Brenti, and Roichman in \cite{adin2001descent} to prove a Carlitz-type formula for a joint Euler--Mahonian distribution in Type $B_n$.

The statistics $\maj_A,\ \maj_B$, and $\nneg$ are related as follows.

\begin{lemma}\label{lem:majB}
  For any signed permutation $(\veps,w)$,
  \[ \maj_B(\veps,w) = \maj_A(\veps,w) + \nneg(\veps,w). \]
\end{lemma}

\begin{proof}
  Let $(\veps,w)=w_1\cdots w_n$ in one--line notation. Set $w_0=0$ and $w_{n+1}=n+1$. Then for $i\in\{0,1,\ldots,n\}$, $i$ is an A--descent if $w_i >_A w_{i+1}$ and $i$ is a B--descent if $w_i >_B w_{i+1}$. Let $X$ be the set of A--descents and $Y$ be the set of B--descents of $(\veps,w)$. There is a bijection $\phi:X\ra Y$ where $\phi(i)=i$ if $w_i$ and $w_{i+1}$ have the same sign, and $\phi(i)=\min\{j\mid j>i,\ w_{j+1}>0\}$ if $w_i$ and $w_{i+1}$ have different signs. We observe the identity
  \[ \sum_{i\in X}\phi(i)-i = \nneg(\veps,w), \]
  from which the lemma follows.
\end{proof}

For $X\subseteq[n]$, let $q_X=\prod_{i\in X}q_i$. We prove the following identity.

\begin{theorem}\label{thm:signed_general}
For all $n\geq 1$,
\[ \det\left(p^{\maj_A(gh^{-1})}q_{\Neg(gh^{-1})}\right)_{g,h\in B_n} = \prod_{k=1}^n(1-q_k^{2k})^{n!2^{n-1}/k}\prod_{k=2}^{n}(1-p^k)^{n!2^n (k-1)/k}. \]
\end{theorem}

Specializing Theorem~\ref{thm:signed_general} gives the following identities, which are Conjectures~46,~47, and~50 in \cite{krattenthaler2005advanced}.

\begin{corollary}\label{cor:signed_spec}
  If $n\geq 1$, then
  \begin{align*}
    \det\left(p^{\maj_A(gh^{-1})}q^{\nneg(gh^{-1})}\right)_{g,h\in B_n} &= \prod_{k=1}^n(1-q^{2k})^{n!2^{n-1}/k}\prod_{k=2}^{n}(1-p^k)^{n!2^n (k-1)/k},\\
    \det\left(q^{\maj_B(gh^{-1})}\right)_{g,h\in B_n} &= \prod_{k=1}^n(1-q^{2k})^{n!2^{n-1}/k}\prod_{k=2}^{n}(1-q^k)^{n!2^n (k-1)/k},\ \mbox{and}\\
    \det\left(p^{\maj_A(gh^{-1})}q^{\sneg(gh^{-1})}\right)_{g,h\in B_n} &= \prod_{k=1}^n(1-q^{2k^2})^{n!2^{n-1}/k}\prod_{k=2}^{n}(1-p^k)^{n!2^n (k-1)/k}.
  \end{align*}
\end{corollary}

\begin{proof}
  For $g\in B_n$, $q_{\Neg(g)}$ specializes to $q^{\nneg(g)}$ by setting $q_i=q$ for all $i$. This gives the first identity. The second follows from the first by setting $p=q$. For the third identity, we observe that $q_{\Neg(g)}$ specializes to $q^{\sneg(g)}$ by setting $q_i=q^i$ for all $i$.
\end{proof}

To prove Theorem~\ref{thm:signed_general}, we construct a certain basis for $B_n$. This basis is not perfect for $n\geq 2$, but it is ``close enough'' for our purposes.

Let $\veps^{(k)}\in\{1,-1\}^n$ where $(\veps^{(k)})_j=-1$ if $k=j$ and $(\veps^{(k)})_j=1$ if $k\neq j$. We again let $t_k=(k,k-1,\ldots,1)$ be a $k$-cycle. Let $s_k=(\veps^{(k)},t_k)$ and $u_k=(\mathbf{1},t_k)$ be signed and unsigned versions of $t_k$, respectively.

\begin{lemma}\label{lem:sneg_basis}
  Let $n\geq 1$ be given. The sequence $(s_1,\ldots,s_n,u_n,u_2)$ is a basis of $B_n$. In particular, every signed permutation $g$ may be uniquely expressed in the form $g=s_1^{d_1}\cdots s_n^{d_n}u_n^{c_n}\cdots u_2^{c_2}$ where $0\leq d_k < 2$ and $0\leq c_k < k$ for all $k$. Moreover, $\maj_A(g)=\maj(t_n^{c_n}\cdots t_2^{c_2})$ and $\Neg(g)=\{i\mid d_i=1\}$.
\end{lemma}

\begin{proof}
  Let $g=(\veps,w)\in B_n$. For $v\in\Sfrak_n$, we have $(\veps,w)(\mathbf{1},v)=(\veps,wv)$. That is, right multiplication by an element $(\mathbf{1},v)$ rearranges the positions of the signed integers in $g$ without changing the set of signed integers present.

  Let $h=(\veps,u)$ be the rearrangement of signed integers in $g$ in increasing order relative to $<_A$. Then $h$ is the unique element in the left coset $g\Sfrak_n$ such that $\maj_A(h)=0$. Furthermore, for $v\in\Sfrak_n$, we have $\maj_A(\veps,uv)=\maj(v)$. In particular, $\maj_A(\veps,w) = \maj(u^{-1}w)$. By Lemma~\ref{lem:maj_basis}, there is a unique factorization $u^{-1}w=t_n^{c_n}\cdots t_2^{c_2}$ where $0\leq c_k<k$ for all $k$, and $\maj(u^{-1}w)=c_n+\cdots+c_2$.

  We have seen that the set $T=\{h\in B_n\mid \maj_A(h)=0\}$ is a left transversal to $\Sfrak_n$ in $B_n$. Since $[B_n:\Sfrak_n]=2^n$, we have $|T|=2^n$. To complete the proof, we show that each element $h\in T$ is uniquely expressible in the form $h=s_1^{d_1}\cdots s_n^{d_n}$ with $0\leq d_i<2$ for all $i$, and $\Neg(h)=\{i\mid d_i=1\}$.

  Let $0\leq d_i<2$ for all $i$, and let $(\veps,u)=s_1^{d_1}\cdots s_{n-1}^{d_{n-1}}$. Then $u$ fixes $n$, and by induction, we may assume $\maj_A(\veps,u)=0$ and $\Neg(\veps,u)=\{i\mid d_i=1, i<n\}$. But $s_1^{d_1}\cdots s_n^{d_n}=(\veps,u)(\veps^{(n)},t_n)=(\veps+\veps^{(n)},ut_n)$. Multiplying $u$ on the right by $t_n$ rotates the values of $u$ and puts $n$ at the beginning. Since $n\in\Neg(s_1^{d_1}\cdots s_n^{d_n})$, the signed values of $s_1^{d_1}\cdots s_n^{d_n}$ are still in increasing order, i.e. $\maj_A(s_1^{d_1}\cdots s_n^{d_n})=0$.

  Hence, $\{s_1^{d_1}\cdots s_n^{d_n}\mid \forall i,\ 0\leq d_i<2\}\subseteq T$. Since both sets contain $2^n$ elements, they must be equal.
\end{proof}

\begin{proof}[Proof of Theorem~\ref{thm:signed_general}]
  Let $\alpha=\sum_{g\in B_n} p^{\maj_A(g)}q_{\Neg(g)}\cdot g$. If $\phi_{\reg}$ is the regular representation of $B_n$, then
  \[ \Delta_{\phi_{\reg}}(\alpha)= \det\left(p^{\maj_A(gh^{-1})}q_{\Neg(gh^{-1})}\right)_{g,h\in B_n}. \]

  By Lemma~\ref{lem:sneg_basis}, we obtain a factorization
  \[ \alpha = (1+q_1s_1)\cdots(1+q_ns_n)(1+pu_n+\cdots+p^{n-1}u_n)\cdots(1+pu_2). \]
  Therefore,
  \begin{align*}
    \Delta_{\phi_{\reg}}(\alpha) &= \prod_{k=1}^n\Delta_{\phi_{\reg}}(1+q_1s_1) \prod_{k=2}^n\Delta_{\phi_{\reg}}(1+pu_k+\cdots+p^{k-1}u_k)\\
    &= \prod_{k=1}^n\theta_{\phi_{\reg}(s_k)}(-q_k) \prod_{k=2}^n\frac{(1-p^k)^{|B_n|}}{\theta_{\phi_{\reg}(u_k)}(p)}
  \end{align*}

  The order of $u_k=(\mathbf{1},t_k)$ is $k$ and the order of $s_k=(\veps^{(k)},t_k)$ is $2k$. Hence,
  \[ \Delta_{\phi_{\reg}}(\alpha) = \prod_{k=1}^n(1-(-q_k)^{2k})^{n!2^n/(2k)} \prod_{k=2}^n(1-p^k)^{n!2^n(1-1/k)}. \]
\end{proof}

\section*{Acknowledgements}

The authors thank Victor Reiner and Volkmar Welker for initial discussions on this project. They also thank Martin Rubey for his help in investigating related identities for a wide variety of statistics on permutations. The mathematical software SageMath and Mathematica were used to investigate the identities considered in this paper. C.\ Smyth was supported by Simons Collaboration Grant number 360486 during the preparation of this work.

\bibliographystyle{alpha}
\bibliography{bib_det}{}

\end{document}